\documentclass[a4paper]{amsart}
\usepackage{graphicx}
\usepackage{amssymb}
\usepackage{amsmath}
\usepackage{amsthm}
\usepackage{amscd}
\usepackage{color}
\usepackage{colordvi}
\usepackage[all,2cell]{xy}

\UseAllTwocells \SilentMatrices
\newtheorem{thm}{Theorem}[section]

\newtheorem{cor}[thm]{Corollary}
\newtheorem{lem}[thm]{Lemma}

\newtheorem{prop}[thm]{Proposition}
\theoremstyle{definition}

\theoremstyle{remark}
\newtheorem{rem}[thm]{\bf Remark}
\numberwithin{equation}{section}

\begin{document}
\title[Gorenstein homological dimension and some invariants of groups]{Gorenstein homological dimension and some invariants of groups}

\author[Wei Ren, Gang Yang] {Wei Ren, Gang Yang}

\thanks{}
\thanks{}
\date{\today}

\thanks{E-mail: wren$\symbol{64}$cqnu.edu.cn; yanggang$\symbol{64}$mail.lzjtu.cn}

\maketitle

\dedicatory{}%
\commby{}%

\begin{abstract}
For any group $G$, the Gorenstein homological dimension ${\rm Ghd}_RG$ is defined to be the Gorenstein flat dimension of the coefficient ring $R$, which is considered as an $RG$-module with trivial group action. We prove that ${\rm Ghd}_RG < \infty$ if and only if the Gorenstein flat dimension of any $RG$-module is finite, if and only if there exists an $R$-pure $RG$-monic $R\rightarrow A$ with $A$ being $R$-flat and ${\rm Ghd}_RG = {\rm fd}_{RG}A$, where $R$ is a commutative ring with finite Gorenstein weak global dimension. As applications, properties of ${\rm Ghd}$ on subgroup, quotient group, extension of groups as well as Weyl group are investigated. Moreover, we compare the relations between some invariants such as ${\rm sfli}RG$, ${\rm silf}RG$, ${\rm spli}RG$, ${\rm silp}RG$, and Gorenstein projective, Gorenstein flat and PGF dimensions of $RG$-modules; a sufficient condition for Gorenstein projective-flat problem over group rings is given.
\end{abstract}

\section{Introduction}

Let $G$ be any group. Recall that the cohomological dimension ${\rm cd}_\mathbb{Z}G$ and homological dimension ${\rm hd}_\mathbb{Z}G$ are defined as the projective and flat dimension of the $\mathbb{Z}G$-module $\mathbb{Z}$ respectively, where $G$ acts trivially. Studying groups through these dimensions arose from both topological and algebraic sources, and has a long history in group theory.

Let $R$ be a commutative ring of coefficients. For any group $G$, the Gorenstein cohomological dimension ${\rm Gcd}_RG$ \cite{BDT09, ET18}, and the Gorenstein homological dimension ${\rm Ghd}_RG$ \cite{ABHS11, LR}, are defined as the Gorenstein projective and Gorenstein flat dimension of the trivial $RG$-module $R$, respectively. The notions of Gorenstein projective, injective and flat modules, which have the origin dating back to the study of $G$-dimension by Auslander and Bridger \cite{AB69} in 1960s, now form the basis of a version of relative homological algebra known as Gorenstein homological algebra \cite{EJ00}. The Gorenstein projective (resp. Gorenstein flat) dimension generalizes the projective (resp. flat) dimension, in the sense that whenever the latter is finite, then they coincide.

For any group $G$, there is an interesting result \cite[Theorem 2.7]{BDT09} that ${\rm Gcd}_\mathbb{Z}G < \infty$ if and only if there is a $\mathbb{Z}$-split $\mathbb{Z}G$-exact sequence $0\rightarrow \mathbb{Z} \rightarrow A$, where $A$ is a $\mathbb{Z}$-projective ($\mathbb{Z}$-free) $\mathbb{Z}G$-module such that the projective dimension ${\rm pd}_{\mathbb{Z}G}A$ of $A$ is finite; in particular,  ${\rm pd}_{\mathbb{Z}G}A = {\rm Gcd}_{\mathbb{Z}}G$. This was generalized in \cite[Theorem 1.7]{ET18} from the coefficient ring of integers $\mathbb{Z}$ to any commutative ring with finite global dimension.

First, we are motivated to characterize the finiteness of ${\rm Ghd}_RG$, and its relations with Gorenstein flat dimension of modules over the group ring $RG$. However, the arguments are not simply analogous to those for Gorenstein cohomological dimension ${\rm Gcd}_RG$. There are some evidences for this phenomenon. Unlike the well-known fact that every projective module is flat, the relation between Gorenstein projective and Gorenstein flat modules is not fully understood, and it is still an open problem whether Gorenstein projective modules are Gorenstein flat. There is a basic and important property that the class of Gorenstein projective modules is closed under extensions; see for example \cite[Theorem 2.5]{Hol04}. However, the analogous result for Gorenstein flat modules is not easy to prove. In \cite[Theorem 3.7]{Hol04}, an additional assumption that the base ring is coherent is needed. Until recently, the assertion was proved affirmatively over any ring by \cite[Corollary 4.12]{SS20}, where a notion of PGF-modules was invented.

In fact, projectivity connects closely to the splitness, however, flatness is a bit inseparable from purity. We succeed in getting the following, which extend the aforementioned results in \cite{BDT09, ET18} to the case of flatness.
Let $R$ be a commutative ring such that the supremum of flat dimension (length) of injective $R$-modules ${\rm sfli}R$ is finite. We show in Theorem \ref{thm:fGhd} that the following conditions are equivalent:
\begin{enumerate}
\item[(1)] ${\rm Ghd}_{R}G < \infty$.
\item[(2)] There exists an $R$-pure $RG$-exact sequence $0\rightarrow R\rightarrow A$, where $A$ is an $R$-flat $RG$-module of finite flat dimension.
\item[(3)] Every $RG$-module has finite Gorenstein flat dimension.
\end{enumerate}
In this case, we have an equality ${\rm Ghd}_RG = {\rm fd}_{RG}A$, and the following inequalities
\begin{center}${\rm sfli}R \leq {\rm sfli}RG = {\rm G.wgldim}RG \leq {\rm Ghd}_RG + {\rm sfli}R,$\end{center}
where ${\rm G.wgldim}RG$ denotes the Gorenstein weak global dimension of the group ring $RG$.
In particular, we show in Corollary \ref{cor:Gfd<+} that if ${\rm sfli}R < \infty$, then ${\rm Gfd}_{RG}M<\infty$ for any $RG$-module $M$ if and only if ${\rm Ghd}_RG = {\rm Gfd}_{RG}R <\infty$ for the trivial $RG$-module $R$; moreover, $\mathrm{Gfd}_{RG}M$ is bounded by ${\rm Ghd}_RG + {\rm Gfd}_RM$, i.e. $\mathrm{Gfd}_{RG}M \leq {\rm Ghd}_RG + {\rm Gfd}_RM$.

It is worth to remark that the assumption ${\rm sfli}R < \infty$ on the coefficient ring $R$ is not too restrictive.
It follows from \cite[Theorem 5.3]{Emm12} that ${\rm sfli}R < \infty$ if and only if every $R$-module has finite Gorenstein flat dimension. In this case, $R$ is called a ring with finite Gorenstein weak global dimension. In fact, we frequently concern modules with finite, rather than infinite, Gorenstein flat dimension. Rings with finite weak global dimension are strictly contained in those of finite Gorenstein weak global dimension. The most interesting examples of coefficient rings in dealing with applications of group rings in geometry and representation theory, such as the ring of integers $\mathbb{Z}$, the field of rationals $\mathbb{Q}$ and finite fields, are all of finite ``sfli''.

For further applications, we strengthen the role of Gorenstein homological dimension of groups by proving the following properties. As shown in Section 4, the arguments heavily rely on the above $R$-pure monomorphism $R\rightarrow A$ . More precisely, for any commutative ring $R$ with ${\rm sfli}R < \infty$ and any group $G$, we show that
\begin{enumerate}
  \item[(I)] ${\rm Ghd}_RH \leq {\rm Ghd}_RG$ for any  subgroup $H$ of $G$ (see Proposition \ref{prop:GhdH<G}).
  \item[(II)] If $1\to H\rightarrow G\rightarrow L\rightarrow 1$ is an extension of groups, then  ${\rm Ghd}_RG \leq {\rm Ghd}_RH + {\rm Ghd}_RL$ (see Proposition \ref{popo:GhdExt}).
  \item[(III)] ${\rm Ghd}_RG = {\rm Ghd}_R(G/H)$ for any finite normal subgroup $H$ of $G$ (see Proposition \ref{prop:GhdL=GhdG}).
 \item[(IV)] If $H$ is a finite subgroup of $G$, and $N_G(H)$ its normalizer in $G$, then ${\rm Ghd}_RW\leq {\rm Ghd}_RG$ for the Weyl group $W = N_G(H)/H$ (see Corollary \ref{cor:GhdW<G}).
\end{enumerate}
In particular, we generalize \cite[Proposition 4.11]{ABHS11} by (1) immediately: ${\rm Ghd}_\mathbb{Z}H \leq {\rm Ghd}_\mathbb{Z}G$, where $H$ was assumed to be a subgroup of $G$ of finite index in \cite{ABHS11}.

Moreover, we compare Gorenstein homological dimension of groups with the generalized homological dimension of groups introduced by Ikenaga \cite[III, Definition]{Ike84}; see details in Proposition \ref{prop:equ1} and \ref{prop:-sfliRG+}. In particular, we obtain in Corollary \ref{cor:equ2} that ${\rm Ghd}_\mathbb{Z}G < \infty$ if and only if $\underline{{\rm hd}}_\mathbb{Z}G < \infty$, and furthermore, ${\rm Ghd}_\mathbb{Z}G  = \underline{{\rm hd}}_\mathbb{Z}G$ in this case.

In Section 5, we are devoted to study the relation between ${\rm Ghd}_RG$ and some invariants such as ${\rm silp}RG$, ${\rm spli}RG$ and ${\rm sfli}RG$, which are closely related to the theory of Gorenstein projective, Gorenstein flat and PGF dimensions of modules; see for example \cite{CET21, DE22, Emm12}. First, we observe that if ${\rm Ghd}_RG<\infty$, then ${\rm sfli}R < \infty$ if and only if ${\rm sfli}RG < \infty$; moreover, if both ${\rm Ghd}_RG$ and ${\rm sfli}R$ are finite, then every Gorenstein projective $RG$-module is a PGF-module, and furthermore is a Gorenstein flat module; see Proposition \ref{prop:obs}.

For any commutative Gorenstein ring $R$ and any group $G$, it follows from \cite[Theorem 2.4]{GG87} that ${\rm spli}RG < \infty$ implies ${\rm silp}RG < \infty$. For the finiteness of ${\rm sfli}RG$ and ${\rm silf}RG$, we show in Proposition \ref{prop:silfRG<} that ${\rm silf}RG < \infty$ if and only if both ${\rm sfli}RG < \infty$ and ${\rm splf}RG < \infty$, that is, the comparison between ${\rm sfli}RG$ and ${\rm silf}RG$ is essentially about the relation between the projective and flat dimensions of modules. Then, we can show that if ${\rm silf}RG < \infty$, then ${\rm Ghd}_RG < \infty$ and every Gorenstein projective $RG$-module is Gorenstein flat; moreover, for any Gorenstein flat $RG$-module $M$, ${\rm PGF}\text{-}{\rm dim}_{RG}M = {\rm Gpd}_{RG}M < \infty$; see Corollary \ref{cor:GP-F(RG)}. In particular, if $R$ is commutative Gorenstein and $G$ is a group such that $RG$ is a left noetherian ring, we show in Proposition \ref{prop:silf=sfli} that ${\rm sfli}RG < \infty$ if and only if ${\rm silf}RG < \infty$; in this case ${\rm sfli}RG = {\rm silf}RG$. It is trivially true that for any finite group, its integral group is noetherian. Remark that for any virtually polycyclic group $G$, also named a polycyclic-by-finite group, the integral group ring $\mathbb{Z}G$ is noetherian; see details in \cite{Hall}. Finally, we show in Proposition \ref{prop:Gwd<Ggd} that there is an inequality ${\rm G.wgldim}RG \leq {\rm PGF}\text{-}{\rm gldim}RG = {\rm G.gldim}RG\leq {\rm G.wgldim}RG + {\rm splf}RG$ for any commutative ring $R$ and any group $G$.

\section{Preliminaries}

In this section, we recall some notions and facts which will be needed in the following.

\subsection*{Gorenstein flat dimension of modules}
Let $\Lambda$ be an associative ring with identity. Recall that $\mathbf{F} = \cdots\rightarrow F_{1}\rightarrow F_{0}\rightarrow F_{-1}\rightarrow\cdots$
is called a {\em totally acyclic complex of flat modules}, provided it is acyclic with each $F_i$ being a flat left $\Lambda$-module, and for any injective right $\Lambda$-module $I$, the complex remains acyclic after applying $I\otimes_\Lambda-$. A left $\Lambda$-module $M$ is called {\em Gorenstein flat} \cite{EJT93}, if there exists a totally acyclic complex of flat modules $\mathbf{F}$, such that $M \cong \mathrm{Ker}(F_0\rightarrow F_{-1})$.

Let $M$ be any left $\Lambda$-module. The {\em Gorenstein flat dimension} of $M$, denoted by $\mathrm{Gfd}_\Lambda M$, is defined by declaring that ${\rm Gfd}_{\Lambda}M \leq n$ if and only if $M$ has a Gorenstein flat resolution $0\rightarrow P_{n}\rightarrow\cdots\rightarrow P_{1}\rightarrow P_{0}\rightarrow M\rightarrow 0$ of length $n$, where each $P_i$ is a Gorenstein flat left $\Lambda$-module. The {\em Gorenstein weak global dimension} \cite{BM10} of $\Lambda$, denoted by ${\rm G.wgldim}\Lambda$, is defined as the supremum of Gorenstein flat dimensions of all left $\Lambda$-modules.

It was proved in \cite[Theorem 3.7]{Hol04} that for a right coherent ring, the class of Gorenstein flat left modules is closed under extensions and direct summands. This result was extended and generalized to any associative ring by \cite[Corollary 4.12]{SS20}. This basic property is crucial for studying homology of Gorenstein flat modules, however, it is not easy to prove. Thanks to
\cite[Corollary 4.12]{SS20}, now we can remove the assumption of coherent rings in many situations when dealing with Gorenstein flat modules and Gorenstein flat dimension of modules, see for example \cite[Theorem 3.14]{Hol04}.

\subsection*{Modules over group rings}
Let $G$ be a group, and $R$ be a commutative ring. A module over the group ring $RG$ is simply an $R$-module $M$ together with an action of $G$ on $M$. In particular, $R$ is an $RG$-module with trivial $G$-action, that is $gr = r$ for any $g\in G$ and any $r\in R$.

For an $RG$-module $M$, the {\em module of invariants of $M$}, denoted by $M^G$, is defined as the largest submodule of $M$ on which $G$ acts trivially, that is, $M^G:= \{m\in M | gm = m \text{ for all }g\in G\}$.  Analogously, the {\em module of coinvariants of $M$}, denoted by $M_G$, is defined to be the largest quotient of $M$ on which $G$ acts trivially.

Let $H$ be any subgroup of $G$. There exist simultaneously an induction functor $\mathrm{Ind}_H^G = RG\otimes_{RH}-$ and a coinduction functor $\mathrm{Coind}_H^G = \mathrm{Hom}_{RH}(RG, -)$ from the category of $RH$-modules ${\rm Mod}(RH)$ to the category of $RG$-modules ${\rm Mod}(RG)$. We denote the restriction functor from ${\rm Mod}(RG)$ to ${\rm Mod}(RH)$ by ${\rm Res}_H^G$. In particular, for any $RG$-module $M$, ${\rm Res}_H^GM$ is exactly the underlying $R$-module if we consider the subgroup $H$ formed only by the identity element of $G$. It is clear that $({\rm Ind}_H^G, {\rm Res}_H^G)$ and $({\rm Res}_H^G, {\rm Coind}_H^G)$ are adjoint pairs of functors. To simplify the notations, in the following we will denote ${\rm Res}_H^GM$ by $M$ if there is no risk of ambiguity.

\subsection*{Gorenstein homological dimension of groups}
Recall that for any group $G$, the {\em Gorenstein homological dimension} of $G$ is defined to be the Gorenstein flat dimension of the trivial $\mathbb{Z}G$-module $\mathbb{Z}$; see \cite[Definition 4.5]{ABHS11}. Analogously, we may define the Gorenstein homological dimension of $G$ over any commutative ring $R$, denoted by ${\rm Ghd}_{R}G$, to be the Gorenstein flat dimension of the trivial $RG$-module $R$; see \cite[Definition 2.5]{LR}. It follows from \cite[Corollary 2.8]{LR} that if the coefficient ring $R$ is $\mathbb{Z}$-torsion-free, then ${\rm Ghd}_RG$ is a refinement of ${\rm Ghd}_\mathbb{Z}G$, that is, ${\rm Ghd}_RG \leq {\rm Ghd}_\mathbb{Z}G$.

\subsection*{Invariants silp, spli, silf and sfli}
The Gorenstein flat dimension is closely related to some homological invariants. Recall that in connection with the existence of complete cohomological functors in the category of left $\Lambda$-modules, Gedrich and Gruenberg have defined in \cite{GG87} the invariant ${\rm silp}\Lambda$ as the supremum of the injective length (dimension) of projective left $\Lambda$-modules, and the invariant ${\rm spli}\Lambda$ as the supremum of the projective length (dimension) of injective left $\Lambda$-modules. Analogously, we use ${\rm silf}\Lambda$ to denote the supremum of the injective length (dimension) of flat left $\Lambda$-modules, and use ${\rm sfli}\Lambda$ to denote the supremum of the flat length (dimension) of injective left $\Lambda$-modules.

Note that if ${\rm sfli}\Lambda$ is finite, then any acyclic complex of flat left $\Lambda$-modules is totally acyclic. It follows from \cite[Theorem 5.3]{Emm12} and \cite[Corollary 1.5]{CET21} that ${\rm G.wgldim}\Lambda = {\rm G.wgldim}\Lambda^{op}$ is finite if and only if ${\rm sfli}\Lambda = {\rm sfli}\Lambda^{op}$ is finite; moreover, in this case one has ${\rm G.wgldim}\Lambda =  {\rm sfli}\Lambda < \infty$, and $\Lambda$ is called a {\em ring with finite Gorenstein weak global dimension}.

\section{Finiteness of Gorenstein homological dimension of groups}

Throughout the paper, $R$ is assumed to be a commutative ring, $G$ is a group. All modules over the group ring $RG$ are left $RG$-modules unless stated otherwise.

In this section, we intend to characterize the finiteness of ${\rm Ghd}_RG$, which is equivalent to that every $RG$-module has finite Gorenstein flat dimension; see Theorem \ref{thm:fGhd}. We begin with the following observation.

\begin{lem}\label{lem:R-Gf}
If ${\rm sfli}R < \infty$, then any Gorenstein flat $RG$-module is also a Gorenstein flat $R$-module.
\end{lem}

\begin{proof}
Let $M$ be a Gorenstein flat $RG$-module. There is a totally acyclic complex of flat $RG$-modules $\cdots\rightarrow F_{1}\rightarrow F_{0}\rightarrow F_{-1}\rightarrow\cdots$  such that $M\cong \mathrm{Ker}(F_0\rightarrow F_{-1})$. Since any flat $RG$-module is also $R$-flat, by restricting this totally acyclic complex, we get an acyclic complex of flat $R$-modules. Noting that ${\rm sfli}R < \infty$, every acyclic complex of flat $R$-modules is totally acyclic. Hence, $M$ is also Gorenstein flat as an $R$-module.
\end{proof}

\begin{prop}\label{prop:PureMonic}
Assume ${\rm sfli}R < \infty$. Let $M$ be an $RG$-module with ${\rm Gfd}_{RG}M < \infty$. If the flat dimension ${\rm fd}_RM$ of the underlying $R$-module is finite, then there exists an $R$-pure $RG$-exact sequence $0\rightarrow M\rightarrow N \rightarrow L\rightarrow 0$, for which $L$ is an $R$-flat $RG$-module and ${\rm Gfd}_{RG}M = {\rm fd}_{RG}N$.
\end{prop}

\begin{proof}
Let ${\rm Gfd}_{RG}M  = n$. It follows from \cite[Theorem 3.23]{Hol04} that there exists an exact sequence $0\rightarrow K\rightarrow X\rightarrow M\rightarrow 0$, where $X$ is a Gorenstein flat $RG$-module, and ${\rm fd}_{RG}K = n-1$. For $X$, there is an exact sequence of $RG$-modules
$0\rightarrow X\rightarrow F\rightarrow L\rightarrow 0$, where $F$ is flat and $L$ is Gorenstein flat. We consider the following pushout of $X\rightarrow M$ and $X\rightarrow F$:
$$\xymatrix{ & & 0\ar[d] & 0\ar[d] \\
0 \ar[r] &K \ar@{=}[d] \ar[r] & X \ar[d]\ar[r] &M \ar@{-->}[d]\ar[r] &0 \\
0 \ar[r] &K \ar[r] & F \ar@{-->}[r] \ar[d] &N \ar[r]\ar[d] & 0\\
&  & L \ar[d] \ar@{=}[r] & L\ar[d]\\
&  & 0 & 0
  }$$
From the middle row we infer that ${\rm fd}_{RG}N = {\rm fd}_{RG}K + 1 = n$, which also implies the finiteness of ${\rm fd}_RN$. In view of the assumption ${\rm fd}_RM < \infty$, we infer from the right column that ${\rm fd}_RL<\infty$. For the Gorenstein flat $RG$-module $L$, it follows from Lemma \ref{lem:R-Gf} that $L$ is also Gorenstein flat as an $R$-module. Furthermore, we imply that $L$ is $R$-flat since flat dimension of any Gorenstein flat module is either zero or infinity; see \cite[Corollary 10.3.4]{EJ00}. Hence, the exact sequence of $RG$-modules $0\rightarrow M\rightarrow N\rightarrow L\rightarrow 0$ is $R$-pure exact. This completes the proof.
\end{proof}

\begin{cor}\label{cor:PureMonic}
Assume ${\rm sfli}R < \infty$. If ${\rm Ghd}_{R}G$ is finite, then there exists an $R$-pure $RG$-exact sequence $0\rightarrow R\rightarrow A$, where $A$ is an $R$-flat $RG$-module such that ${\rm Ghd}_RG = {\rm fd}_{RG}A$.
\end{cor}

Let $M$ and $N$ be left $RG$-modules. It is worth to note that there are two different ways to define the $RG$-module structures for the tensor product $M\otimes_R N$. On the one hand, $M\otimes_{R}N$ is an $RG$-module where $G$ acts diagonally, i.e. $g(m\otimes n) = gm \otimes gn$, but on the other hand, $M\otimes_{R}N$ is an induced left $RG$-module by $g(m\otimes n) = gm \otimes n$.

\begin{lem}\label{lem:RG-antisym}
For any left $RG$-modules $M$ and $N$, there is an isomorphism of induced left $RG$-modules $M\otimes_R N\cong N\otimes_R M$.
\end{lem}

\begin{proof}
Since $R$ is assumed to be commutative, by using the anti-automorphism $g\rightarrow g^{-1}$ of $G$, we can regard any left $RG$-module $X$ as a right $RG$-module by setting $gx = xg^{-1}$ for any $g\in G$ and $x\in X$. Thus, we may regard $M\otimes_R N$ as both a left $RG$-module with $G$-actions induced from $_{RG}M\otimes_R-$ by setting $g(m\otimes n) = (gm)\otimes n$, and as a right $RG$-module with $G$-actions induced from $-\otimes_R N_{RG}$ by setting $(m\otimes n)g^{-1} =m\otimes (ng^{-1}) = m\otimes gn$.

Note that for any $g\in G$ and any $m\otimes n\in M\otimes_{R}N$, one has $gm\otimes n = g(m\otimes n) = (m\otimes n)g^{-1} = m\otimes gn$. Similarly, for any $n\otimes m\in N\otimes_R M$, $gn\otimes m = g(n\otimes m) = (n\otimes m)g^{-1} = n\otimes gm$. It is direct to check that the map $\varphi: M\otimes_R N \rightarrow N\otimes_R M$ given by $m\otimes n\mapsto n\otimes m$ is an isomorphism of the induced left $RG$-modules.
\end{proof}

\begin{rem}
Since $R$ is assumed to be commutative, for any $RG$-module, one can avoid considering both left and right modules by using the anti-automorphism $g\rightarrow g^{-1}$ of $G$. Thus, for any two left $RG$-modules $M$ and $N$, the tensor product $M\otimes_{RG}N$ makes sense by introducing the relations $gm\otimes n = mg^{-1}\otimes n = m\otimes g^{-1}n$.

For any $RG$-modules $M$ and $N$, By replacing $n$ in the above relations by $gn$, we get that $gm\otimes gn = m\otimes n$. Hence, $M\otimes_{RG}N = (M\otimes_{R}N)_{G}$, where $G$ acts diagonally on $M\otimes_{R}N$. Moreover, there is an isomorphism $M\otimes_{RG}N\cong N\otimes_{RG}M$, that is, the bifunctor $-\otimes_{RG}-$ is commutative.
\end{rem}

\begin{lem}\label{lem:RG-Gf}
Let $\iota: R\rightarrow A$ be an $R$-pure monomorphism of $RG$-modules, where $A$ is $R$-flat. For any $RG$-module $N$, if the induced left $RG$-module $N\otimes_{R}A$ is Gorenstein flat, then so is $N$.
\end{lem}

\begin{proof}
Assume that $N\otimes_{R}A$ is a Gorenstein flat $RG$-module. Then, there is an exact sequence of $RG$-modules
$$0\longrightarrow  N\otimes_{R}A \stackrel{\alpha}\longrightarrow F_0\longrightarrow L\longrightarrow 0,$$
where $F_0$ is flat and $L$ is Gorenstein flat. By applying $N\otimes_{R}-$ to the $R$-pure monomorphism $\iota: R\rightarrow A$, we get a monomorphism of induced left $RG$-modules ${\rm Id}_N\otimes\iota: N = N\otimes_R R\rightarrow N\otimes_R A$. Let $\beta = \alpha({\rm Id}_N\otimes\iota): N\rightarrow F_0$, and consider the following commutative diagram
$$\xymatrix{ & & 0\ar[d] & 0\ar[d] \\
0 \ar[r] &N \ar@{=}[d] \ar[r]^{{\rm Id}_N\otimes \iota\quad} & N\otimes_R A \ar[d]^{\alpha}\ar[r] & N\otimes_R B \ar[d]\ar[r] &0 \\
0 \ar[r] &N \ar[r]^{\beta} & F_0 \ar[r] \ar[d] &N' \ar[r]\ar[d] & 0\\
&  & L \ar[d] \ar@{=}[r] & L\ar[d]\\
&  & 0 & 0
  }$$
where $B = {\rm Coker}\iota$, and $N' = {\rm Coker}\beta$.

Let $E$ be any injective right $RG$-module. It is clear that $${\rm Id}_E\otimes\alpha: E\otimes_{RG}(N\otimes_{R}A) \rightarrow E\otimes_{RG}F_0$$ is monic. Since $\iota: R\rightarrow A$ is an $R$-pure monomorphism, we infer that
$${\rm Id}_E \otimes ({\rm Id}_N\otimes\iota): E\otimes_{RG}N =(E\otimes_{RG}N)\otimes_{R}R  \rightarrow (E\otimes_{RG}N)\otimes_{R} A\cong E\otimes_{RG}(N\otimes_{R} A)$$
is monic. Hence, ${\rm Id}_E\otimes\beta: E\otimes_{RG}N \rightarrow E\otimes_{RG}F_0$ is a monomorphism, and moreover, we infer from the middle row of the above diagram that ${\rm Tor}_1^{RG}(E, N') = 0$.

By applying $-\otimes_{R}A$ to the middle row, we get an exact sequence of induced left $RG$-modules $$0 \longrightarrow N\otimes_{R}A \longrightarrow F_0\otimes_{R}A \longrightarrow N'\otimes_{R}A  \longrightarrow 0,$$
where $F_0 \otimes_{R}A $ is flat and $N\otimes_{R}A$ is assumed to be Gorenstein flat. Since $A$ is a flat $R$-module, there is an isomorphism ${\rm Tor}_i^{RG}(E, N'\otimes_R A)\cong {\rm Tor}_i^{RG}(E, N')\otimes_R A$ for any $i>0$, and moreover, we get ${\rm Tor}_1^{RG}(E,N'\otimes_{R} A) = 0$ from ${\rm Tor}_1^{RG}(E, N') = 0$. Note that by \cite[Corollay 4.12]{SS20}, the class of all Gorenstein flat modules is always closed under extensions. Then, we may get rid of the coherent assumption on rings in \cite[Proposition 3.8]{Hol04}, and obtain that the induced left $RG$-module $N'\otimes_{R}A$ is also Gorenstein flat.

Analogous to the above, for $N'$ we obtain an exact sequence of left $RG$-modules
$0\rightarrow N'\rightarrow F_{-1}\rightarrow N''\rightarrow 0$ with $F_{-1}$ being flat, which remains exact after applying $E\otimes_{RG}-$ for any injective right $RG$-module $E$, and for $N''$ the induced left $RG$-module $N''\otimes_{R}A$ is Gorenstein flat. Proceed in this manner, we obtain an acyclic complex
$$\xymatrix{\mathbf{F}_{\leq 0} = 0\ar[r]^{} & N\ar[r] &F_0 \ar[r]^{}& F_{-1}\ar[r] & F_{-2}\ar[r] & \cdots} $$
with each $F_i$ being a flat left $RG$-module, which remains acyclic after applying $E\otimes_{RG}-$ for any injective right $RG$-module $E$.

Now consider a flat resolution of $N$:
$$\mathbf{F}_{>0} = \xymatrix{\cdots\ar[r]^{} & F_{3}\ar[r] &F_2 \ar[r]^{}& F_1\ar[r] & N\ar[r]&  0.} $$
Since $A$ is $R$-flat, it is clear that the induced left $RG$-modules $F_i\otimes_R A$ are flat, and  $\mathbf{F}_{>0}^{\bullet}\otimes_R A\rightarrow N\otimes_R A$ is a flat resolution of the induced left $RG$-module $N\otimes_R A$. Invoking the assumption that $N\otimes_R A$ is a Gorenstein flat left $RG$-module, we get ${\rm Tor}_i^{RG}(E, N\otimes_{R}A) = 0$ for any injective right $RG$-module $E$ and any $i\geq 1$, and this implies that the complex $E\otimes_{RG}(\mathbf{F}_{>0}^{\bullet}\otimes_R A)\rightarrow E\otimes_{RG}(N\otimes_R A)$
is acyclic. Thus, we infer from the isomorphisms
$$(A\otimes_R E)\otimes_{RG}\mathbf{F}_{>0} \cong A\otimes_R(E\otimes_{RG}\mathbf{F}_{>0}) \cong (E\otimes_{RG}\mathbf{F}_{>0})\otimes_R A\cong E\otimes_{RG}(\mathbf{F}_{>0}\otimes_R A)$$
that the complex $(A\otimes_R E)\otimes_{RG}\mathbf{F}_{>0}$ is acyclic.

By applying $-\otimes_R E$ to the $R$-pure exact sequence $0\rightarrow R\rightarrow A\rightarrow B\rightarrow 0$, we get an exact sequence of induced right $RG$-modules
$$0\longrightarrow E=R\otimes_R E\longrightarrow A\otimes_R E\longrightarrow B\otimes_R E\longrightarrow 0.$$
The sequence is split since $E$ is injective, and moreover, $E$ is a direct summand of $A\otimes_R E$. Thus, as a direct summand of $(A\otimes_R E)\otimes_{RG}\mathbf{F}_{>0}$, the complex $E\otimes_{RG}\mathbf{F}_{>0}$ is acyclic.

Consequently, by pasting $\mathbf{F}_{\leq 0}$ and $\mathbf{F}_{>0}$ we will get a totally acyclic complex of flat $RG$-modules such that $N\cong \mathrm{Ker}(F_0\rightarrow F_{-1})$. This yields that $N$ is a Gorenstein flat $RG$-module. This completes the proof.
\end{proof}

\begin{lem}\label{lem:GF-H2G}
Let $H$ be any subgroup of $G$. For any Gorenstein flat $RH$-module $M$, the induced $RG$-module ${\rm Ind}_H^GM$ is also Gorenstein flat.
\end{lem}

\begin{proof}
Let $M$ be a Gorenstein flat $RH$-module. There is a totally acyclic complex of flat left $RH$-modules
$$\mathbf{F} = \cdots\longrightarrow F_1\longrightarrow F_0\longrightarrow F_{-1}\longrightarrow\cdots$$
such that $M = {\rm Ker}(F_0\rightarrow F_{-1})$. We infer that
${\rm Ind}_H^G\mathbf{F} = RG\otimes_{RH}\mathbf{F}$ is an acyclic complex of flat $RG$-modules such that ${\rm Ind}_H^GM = {\rm Ker}({\rm Ind}_H^GF_0\rightarrow {\rm Ind}_H^GF_{-1})$. For any injective right $RG$-module $E$, it is restricted to be an injective $RH$-module. Then, we infer from the isomorphism
$$E\otimes_{RG}{\rm Ind}_H^G\mathbf{F} = E\otimes_{RG}RG\otimes_{RH}\mathbf{F}\cong E\otimes_{RH}\mathbf{F}$$
that $E\otimes_{RG}{\rm Ind}_H^G\mathbf{F}$ is acyclic. Hence, ${\rm Ind}_H^G\mathbf{F}$ is a totally acyclic complex of flat $RG$-modules, and this implies that ${\rm Ind}_H^GM$ is a Gorenstein flat $RG$-module.
\end{proof}

\begin{lem}\label{lem:IndGF}
Let $M$ be an $RG$-module. If $M$ is Gorenstein flat as an $R$-module, then for any flat $RG$-module $F$, the induced left $RG$-module $F\otimes_{R}M$ is also Gorenstein flat.
\end{lem}

\begin{proof}
Let $H = \{1\}$ be the subgroup formed by the identity element of $G$. Then, the induction functor $\mathrm{Ind}_H^G = RG\otimes_{R}-$ sends every $R$-module to be an $RG$-module. Let $M$ be an $RG$-module, and assume that $M$ is Gorenstein flat as an $R$-module. Since any flat $RG$-module $F$ is also restricted to be a flat $R$-module, we may express it as a colimit of free $R$-modules. Hence, it follows that $F\otimes_{R}M$ is a Gorenstein flat $R$-module, as the subcategory of Gorenstein flat $R$-modules is closed under direct limits; see \cite[Corollary 4.12]{SS20} and \cite[Lemma 3.1]{YL12}. Then, by Lemma \ref{lem:GF-H2G} we infer that the induced $RG$-module ${\rm Ind}_H^G(F\otimes_{R}M)$ is Gorenstein flat.

There is an isomorphism of $RG$-modules $$F\otimes_R{\rm Ind}_H^GM \cong {\rm Ind}_H^G(F\otimes_RM),$$
where $G$ acts diagonally on the left tensor product; see for example \cite[Section III 5]{Bro82}. Note that the diagonal $RG$-module structure of $F\otimes_R{\rm Ind}_H^GM$ coincides with the one induced by $_{RG}F\otimes_{R}-$. As $R$-modules, $M$ is a direct summand of ${\rm Ind}_H^GM$. Then, the induced $RG$-module $F\otimes_RM$ is a direct summand of the induced $RG$-module $F\otimes_R{\rm Ind}_H^GM$. Hence, $F\otimes_{R}M$ is a Gorenstein flat $RG$-module.
\end{proof}

\begin{prop}\label{prop:Gfd<fd}
Let $M$ be any left $RG$-module which is Gorenstein flat as an $R$-module. If there exists an $R$-pure monomorphism of $RG$-modules $\iota: R\rightarrow A$, where $A$ is $R$-flat with $\mathrm{fd}_{RG}A < \infty$, then $\mathrm{Gfd}_{RG}M \leq \mathrm{fd}_{RG}A$.
\end{prop}

\begin{proof}
Let $\mathrm{fd}_{RG}A = n$. There exists an exact sequence of left $RG$-modules
$$0\longrightarrow F_n\longrightarrow \cdots \longrightarrow F_1\longrightarrow F_0\longrightarrow A\longrightarrow 0,$$
where each $F_i$ is flat. Since $A$ is a flat $R$-module and each $F_i$ is restricted to be a flat $R$-module, the sequence is $R$-pure exact.

By applying $-\otimes_{R}M$ to the above $R$-pure exact sequence, we get an exact sequence of induced left $RG$-modules
$$0\longrightarrow F_n\otimes_{R}M\longrightarrow \cdots\longrightarrow F_1\otimes_{R}M \longrightarrow F_0\otimes_{R}M \longrightarrow A\otimes_{R}M\longrightarrow 0.$$
Note that by Lemma \ref{lem:RG-antisym}, there is an isomorphism of induced left $RG$-modules $M\otimes_R A\cong A\otimes_R M$. It follows from Lemma \ref{lem:IndGF} that $F_{i}\otimes_RM$ are Gorenstein flat $RG$-modules, and then ${\rm Gfd}_{RG}(M\otimes_{R}A) = {\rm Gfd}_{RG}(A\otimes_{R}M)\leq n$.

Now we consider a flat resolution of the $RG$-module $M$
$$\cdots \longrightarrow F'_n\longrightarrow \cdots \longrightarrow F'_1\longrightarrow F'_0\longrightarrow M\longrightarrow 0,$$
and let $K = {\rm Ker}(F'_{n-1}\rightarrow F'_{n-2})$.
By applying $-\otimes_{R}A$, we get an exact sequence of induced left $RG$-modules
$$0 \longrightarrow K\otimes_{R}A \longrightarrow F'_{n-1}\otimes_{R}A\rightarrow \cdots \longrightarrow F'_0\otimes_{R}A\longrightarrow M\otimes_{R}A\longrightarrow 0,$$
where all $F'_i\otimes_{R}A$ are flat $RG$-modules. We infer from  ${\rm Gfd}_{RG}(M\otimes_{R}A)\leq n$ that $K\otimes_{R}A$ is a Gorenstein flat left $RG$-module. Hence, it follows from Lemma \ref{lem:RG-Gf} that $K$ is a Gorenstein flat $RG$-module, and then the desired inequality $\mathrm{Gfd}_{RG}M \leq n$ holds.
\end{proof}

\begin{cor}\label{cor:Gfd}
If there exists an $R$-pure monomorphism of $RG$-modules $\iota: R\rightarrow A$, where $A$ is $R$-flat with $\mathrm{fd}_{RG}A < \infty$, then for any $RG$-module $M$, we have
$$\mathrm{Gfd}_{RG}M \leq \mathrm{fd}_{RG}A + {\rm Gfd}_RM \leq \mathrm{fd}_{RG}A + {\rm sfli}R.$$
\end{cor}

\begin{proof}
Let $\mathrm{fd}_{RG}A = n$. The second inequality is obvious since ${\rm G.wgldim}R ={\rm sfli}R$; see \cite[Theorem 5.3]{Emm12} and \cite[Corollary 1.5]{CET21}.

For the first inequality, it suffices to assume that ${\rm Gfd}_RM = m$ is finite. In this case, we may prove the result by induction on $m$. If $m=0$, that is, $M$ is a Gorenstein flat $R$-module, then the assertion is immediate from Proposition \ref{prop:Gfd<fd}.

We now assume $m > 0$, and consider a short exact sequence of $RG$-modules
$$0\longrightarrow K\longrightarrow F\longrightarrow M\longrightarrow 0,$$
where $F$ is flat. Since $F$ is restricted to be a flat $R$-module, we have ${\rm Gfd}_{R}K = m - 1$. Invoking the induction hypothesis, we may conclude that ${\rm Gfd}_{RG}K \leq n + (m-1)$, and hence ${\rm Gfd}_{RG}M\leq n + m$. This completes the proof.
\end{proof}

The following is immediate from Corollary \ref{cor:PureMonic} and \ref{cor:Gfd}.

\begin{cor}\label{cor:Gfd<+}
Let $R$ be a commutative ring with ${\rm sfli}R < \infty$. Then ${\rm Ghd}_RG<\infty$ if and only if ${\rm Gfd}_{RG}M<\infty$ for any $RG$-module $M$. In this case, one has
$$\mathrm{Gfd}_{RG}M \leq {\rm Ghd}_RG + {\rm Gfd}_RM.$$
\end{cor}

Now, we are in a position to state the main result of this section.

\begin{thm}\label{thm:fGhd}
Let $G$ be a group, $R$ a commutative ring with ${\rm sfli}R < \infty$. The following are equivalent:
\begin{enumerate}
\item ${\rm Ghd}_{R}G$ is finite.
\item There exists an $R$-pure $RG$-exact sequence $0\rightarrow R\rightarrow A$, where $A$ is an $R$-flat $RG$-module of finite flat dimension.
\item Any $RG$-module has finite Gorenstein flat dimension.
\item ${\rm sfli}RG$ is finite.
\end{enumerate}
In this case, there is an equality ${\rm Ghd}_RG = {\rm fd}_{RG}A$, and moreover, the following inequalities hold:
$${\rm sfli}R \leq {\rm sfli}RG = {\rm G.wgldim}RG \leq {\rm Ghd}_RG + {\rm sfli}R.$$
\end{thm}

\begin{proof}
The implication (1)$\Longrightarrow$(2) follows by Corollary \ref{cor:PureMonic}, and the implication (2)$\Longrightarrow$(3) follows from Corollary \ref{cor:Gfd}. The implication (3)$\Longrightarrow$(1) is trivial, and (3)$\Longleftrightarrow$(4) follows from \cite[Theorem 5.3]{Emm12}.

For the first inequality ${\rm sfli}R \leq {\rm sfli}RG$, it suffices to assume that ${\rm sfli}RG = n$ is finite. We consider any injective $R$-module $I$. For the subgroup $H = \{1\}$ of $G$, note that ${\rm Coind}_H^GI$ is an injective $RG$-module. Then ${\rm fd}_{RG}{\rm Coind}_H^GI \leq n$, which yields ${\rm fd}_R{\rm Coind}_H^GI \leq n$. Moreover, as $R$-modules $I$ is a direct summand of ${\rm Coind}_H^GI$, and hence ${\rm fd}_RI \leq n$. This implies that ${\rm sfli}R\leq n$.

By \cite[Theorem 5.3]{Emm12} we have ${\rm sfli}R = {\rm G.wgldim}R$ and ${\rm sfli}RG = {\rm G.wgldim}RG$.
Then, the inequality ${\rm G.wgldim}RG \leq {\rm Ghd}_RG + {\rm G.wgldim}R$ follows immediately from Corollary \ref{cor:Gfd<+}.
\end{proof}

\section{Properties of Gorenstein homological dimension of groups}

Using the above characterization on the finiteness of ${\rm Ghd}_RG$, we obtain the following properties of the Gorenstein homological dimensions of groups.

\begin{prop}\label{prop:GhdH<G}
Let $H$ be any subgroup of $G$. If ${\rm sfli}R < \infty$, then ${\rm Ghd}_RH \leq {\rm Ghd}_RG$.
\end{prop}

\begin{proof}
There is nothing to prove if ${\rm Ghd}_RG = \infty$ and hence we may assume that ${\rm Ghd}_RG = n$ is finite. It follows from Corollary \ref{cor:PureMonic} that there exists an $R$-pure $RG$-exact sequence $0\rightarrow R\rightarrow A$ for which $A$ is $R$-flat such that ${\rm fd}_{RG}A = n$. Note that every flat $RG$-module is restricted to be a flat $RH$-module, and $0\rightarrow R\rightarrow A$ is also an exact sequence of $RH$-modules. We infer that ${\rm fd}_{RH}A \leq {\rm fd}_{RG}A$, and moreover, for the trivial $RH$-module $R$, by Proposition \ref{prop:Gfd<fd} we have ${\rm Ghd}_RH = {\rm Gfd}_{RH}R \leq {\rm fd}_{RH}A \leq n$. This completes the proof.
\end{proof}

The following generalizes \cite[Proposition 4.11]{ABHS11}: let $H$ be a subgroup of $G$ of finite index, then ${\rm Ghd}_\mathbb{Z}H \leq {\rm Ghd}_\mathbb{Z}G$. We remove the assumption of finite index therein.

\begin{cor}\label{cor:GhdH<G}
Let $H$ be any subgroup of $G$. Then ${\rm Ghd}_\mathbb{Z}H \leq {\rm Ghd}_\mathbb{Z}G$, and  ${\rm Ghd}_\mathbb{Q}H \leq {\rm Ghd}_\mathbb{Q}G$.
\end{cor}

\begin{prop}\label{popo:GhdExt}
Let $1\rightarrow H\rightarrow G\rightarrow L\rightarrow 1$ be an extension of groups. If ${\rm sfli}R<\infty$, then
${\rm Ghd}_RG \leq {\rm Ghd}_RH + {\rm Ghd}_RL.$
\end{prop}

\begin{proof}
It suffices to assume that both ${\rm Ghd}_RH = m$ and ${\rm Ghd}_RL = n$ are finite. Then, it follows from Corollary \ref{cor:PureMonic} that there exists an $R$-pure $RL$-exact sequence $0\rightarrow R\stackrel{\alpha}\rightarrow A$ for which $A$ is $R$-flat such that ${\rm fd}_{RL}A = {\rm Ghd}_RL = n$.

For the quotient group $L=G/H$, we may consider every $RL$-module as an $RG$-module through the natural morphism of group rings $RG\rightarrow RL$. We claim that for any flat $RL$-module $F$, one has ${\rm Gfd}_{RG}F\leq {\rm Ghd}_RH$. Note that $RL = R[G/H] = {\rm Ind}_H^GR$; see for example \cite[Proposition III 5.6]{Bro82}. Since ${\rm Ghd}_RH = {\rm Gfd}_{RH}R = m$, there is a Gorenstein flat $RH$-resolution of length $m$
$$0\longrightarrow M_m\longrightarrow \cdots\longrightarrow M_1\longrightarrow M_0\longrightarrow R\longrightarrow 0.$$
By applying the induction functor, we get an exact sequence of $RG$-modules
$$0\longrightarrow {\rm Ind}_H^GM_m\longrightarrow \cdots\longrightarrow {\rm Ind}_H^GM_1\longrightarrow {\rm Ind}_H^GM_0\longrightarrow {\rm Ind}_H^GR = RL\longrightarrow 0.$$
By Lemma \ref{lem:GF-H2G} we infer that all ${\rm Ind}_H^GM_i$ are Gorenstein flat $RG$-modules. Then, we get that ${\rm Gfd}_{RG}RL \leq  m$, and moreover, by \cite[Proposition 3.13]{Hol04} we induce that  ${\rm Gfd}_{RG}P \leq  m$ for any free $RL$-module $P$. For any flat $RL$-module $F$, it follows from Lazard-Govorov Theorem that $F = \varinjlim P_i$ is a direct limit of finitely generated free $RL$-modules $P_i$. Furthermore, since the subcategory of Gorenstein flat modules is closed under direct limits (see \cite[Corollary 4.12]{SS20} and \cite[Lemma 3.1]{YL12}), we prove the claim that ${\rm Gfd}_{RG}F = {\rm Gfd}_{RG}(\varinjlim P_i)\leq m={\rm Ghd}_RH$.

Recall that ${\rm fd}_{RL}A = {\rm Ghd}_RL = n$. We consider an exact sequence of $RL$-modules
$$0\longrightarrow F_n\longrightarrow \cdots\longrightarrow F_1\longrightarrow F_0\longrightarrow A\longrightarrow 0$$
in which all $F_i$ are flat. It follows from the above argument that ${\rm Gfd}_{RG}F_i \leq m$ for $0\leq i\leq n$. By a standard argument, we infer from the above exact sequence that ${\rm Gfd}_{RG}A \leq m+n$.

Since $A$ is an $R$-flat module and ${\rm Gfd}_{RG}A \leq m+n$ is finite, it follows from Proposition \ref{prop:PureMonic} that there is an $R$-pure $RG$-exact sequence $0\rightarrow A\stackrel{\beta}\rightarrow B$, where $B$ is an $R$-flat $RG$-module such that ${\rm fd}_{RG}B = {\rm Gfd}_{RG}A \leq m+n$. Let $\iota = \beta\alpha$. Then, we obtain an $R$-pure $RG$-exact sequence $0\rightarrow R\stackrel{\iota}\rightarrow B$, and consequently, it follows from Proposition \ref{prop:Gfd<fd} that
\begin{center}${\rm Ghd}_RG = {\rm Gfd}_{RG}R \leq {\rm fd}_{RG}B \leq m+n.$\end{center} This completes the proof.
\end{proof}

\begin{prop}\label{prop:GhdL=GhdG}
Let $H$ be a finite normal subgroup of $G$, and $R$ be a commutative ring such that ${\rm sfli}R < \infty$. Then ${\rm Ghd}_RG = {\rm Ghd}_R(G/H)$.
\end{prop}

\begin{proof}
By Proposition \ref{popo:GhdExt}, we have ${\rm Ghd}_RG \leq {\rm Ghd}_RH + {\rm Ghd}_R(G/H)$. Since the subgroup $H$ is finite, we infer that ${\rm Ghd}_RH = 0$; see for example \cite[Corollary 3.3]{LR} or \cite[Proposition 4.12]{ABHS11}. In fact, $R$ is a Gorenstein flat $RH$-module whenever $H$ is a finite group. This yields ${\rm Ghd}_RG \leq {\rm Ghd}_R(G/H)$.

Conversely, since the inequality ${\rm Ghd}_R(G/H)\leq {\rm Ghd}_RG$ is obvious if ${\rm Ghd}_RG = \infty$, it suffices to assume ${\rm Ghd}_RG = n$ is finite. In this case, it follows immediately from Corollary \ref{cor:PureMonic} that there exists an $R$-pure $RG$-exact sequence $0\rightarrow R\stackrel{\iota}\rightarrow A$ for which $A$ is $R$-flat such that ${\rm fd}_{RG}A = n$. Since the group $G$ acts trivially on $R$, it implies that ${\rm Im}\iota\subseteq A^G\subseteq A^H$. We may therefore consider the $R[G/H]$-module $A^H$, and the $R$-pure $R[G/H]$-monomorphism $R\rightarrow A^H$.

In the following, we will prove that $A^H$ is $R$-flat and ${\rm fd}_{R[G/H]}A^H \leq n$. Then, the desired inequality ${\rm Ghd}_R(G/H) = {\rm Gfd}_{R[G/H]}R \leq n$ will hold by Proposition \ref{prop:Gfd<fd}.

Since $A$ is $R$-flat with ${\rm fd}_{RG}A < \infty$ and $H$ is a finite group, we infer that $A$ is a flat $RH$-module, and then $A = \varinjlim P_i$ for some finitely generated free $RH$-modules. Since $(RH)^H\cong R$, we infer that $P_i^H$ are finitely generated free $R$-modules. Note that for any $RG$-module $M$, $M^H\cong {\rm Hom}_{RH}(R, M)$ as $R[G/H]$-modules; see for example \cite[pp.56]{Bro82}. Since $H$ is a finite group, $R$ is a finitely presented $RH$-module. Hence, we induce from the isomorphisms
\begin{center}$A^H \cong {\rm Hom}_{RH}(R, \varinjlim P_i) \cong \varinjlim {\rm Hom}_{RH}(R, P_i)\cong \varinjlim P_i^H$\end{center}
that $A^H$ is a flat $R$-module.

For the finite group $H$, $R\rightarrow RH$ is a Frobenius extension of rings, and then ${\rm RHom}_{R}(RH, R)\simeq RH$ in the derived category $\mathbf{D}(RH)$. Hence, we have
\begin{equation*}
\begin{split}
&{\rm RHom}_{RH}(R, RH)\simeq {\rm RHom}_{RH}(R, {\rm RHom}_{R}(RH, R))\\
&\simeq {\rm RHom}_{R}(R\otimes^L_{RH}RH, R) \simeq {\rm RHom}_{R}(R, R) = R,
\end{split}
\end{equation*}
The isomorphisms yield that ${\rm Ext}^1_{RH}(R, RH) = 0$. We remark that the vanishing of ${\rm Ext}^1_{RH}(R, RH)$ can also be obtained by the fact that $R$ is a Gorenstein projective $RH$-module, if one notices that $R\rightarrow RH$ is a Frobenius extension of rings, and the restriction functor from the category of $RH$-modules to the category of $R$-modules is a faithful Frobenius functor; see \cite[Theorem 3.2]{CR22} and \cite[Theorem 2.2]{Ren18}. Moreover, we infer that ${\rm Ext}^1_{RH}(R, F) = 0$ for any flat $RH$-module $F$ since $F$ is a direct limit of some finitely generated free $RH$-modules.

Since ${\rm fd}_{RG}A = n$, there is an exact sequence
$$0\longrightarrow F_n\longrightarrow \cdots\longrightarrow F_1\longrightarrow F_0\longrightarrow A\longrightarrow 0$$
for which $F_i$ are flat $RG$-modules. Analogous to the above argument, we get that $F_i^H\cong {\rm Hom}_{RH}(R, F)$ are flat modules over $R[G/H]$. By restriction, $F_i$ are flat $RH$-modules, and then ${\rm Ext}^1_{RH}(R, F_i) = 0$. By applying
${\rm Hom}_{RH}(R,-)$ to the above sequence, we can obtain an exact sequence of $R[G/H]$-modules
\begin{center}$0\longrightarrow F_n^H\longrightarrow \cdots\longrightarrow F_1^H\longrightarrow F_0^H\longrightarrow A^H\longrightarrow 0.$\end{center}
Hence, ${\rm fd}_{R[G/H]}A^H \leq n$, as expected. This completes the proof.
\end{proof}

The Weyl groups of the subgroups of a given group $G$ are useful tools in the study of actions of $G$ on topological spaces, as these Weyl groups act naturally on the fixed point spaces of the actions. Let $H$ be a subgroup of $G$. We denote by $N_G(H)$ the normalizer of $H$ in $G$.

\begin{cor}\label{cor:GhdW<G}
Let $H$ be a finite subgroup of $G$, and $R$ be a commutative ring such that ${\rm sfli}R < \infty$. Then for the Weyl group $W = N_G(H)/H$, one has ${\rm Ghd}_RW\leq {\rm Ghd}_RG$.
\end{cor}

\begin{proof}
In view of Proposition \ref{prop:GhdH<G}, we have ${\rm Ghd}_RN_G(H)\leq {\rm Ghd}_RG$. It follows from Proposition \ref{prop:GhdL=GhdG} that ${\rm Ghd}_RW = {\rm Ghd}_RN_G(H)$.  Then, we get the inequality ${\rm Ghd}_RW\leq {\rm Ghd}_RG$.
\end{proof}

For a group $G$, recall that Ikenaga introduced the {\em generalized homological dimension} of $G$ over the ring of integers $\mathbb{Z}$; see \cite[III, Definition]{Ike84}. Analogously, we may define the generalized homological dimension of $G$ over any coefficient ring $R$ as follows:
$$\underline{{\rm hd}}_RG = {\rm sup}\{i \in \mathbb{N}~~|~~ {\rm Tor}^{RG}_i(I, M)\neq 0,~~ \exists M ~~R\text{-flat}, ~~\exists ~~I ~~RG\text{-injective}\}.$$

We conclude this section by comparing the Gorenstein homological dimension and generalized homological dimension of groups.

\begin{prop}\label{prop:equ1}
Let $R$ be a commutative ring of coefficients. If ${\rm Ghd}_RG < \infty$, then ${\rm Ghd}_RG  \leq \underline{{\rm hd}}_RG$. Moreover, if ${\rm sfli}R < \infty$, then ${\rm Ghd}_RG  = \underline{{\rm hd}}_RG$.
\end{prop}

\begin{proof}
Assume ${\rm Ghd}_RG = {\rm Gfd}_{RG}R$ is finite. It follows from \cite[Theorem 3.14]{Hol04} that
$${\rm Ghd}_RG = {\rm sup}\{i \in \mathbb{N}~~|~~ {\rm Tor}^{RG}_i(I, R)\neq 0, ~~\exists ~~I ~~RG\text{-injective}\}.$$
Then, we infer from the definition of $\underline{{\rm hd}}_RG$ that ${\rm Ghd}_RG \leq \underline{{\rm hd}}_RG$.

Moreover, if we assume ${\rm sfli}R < \infty$, then every $RG$-module has finite Gorenstein flat dimension; see Theorem \ref{thm:fGhd}. Since both ${\rm Ghd}_RG$ and ${\rm sfli}R$ are assumed to be finite, for any $RG$-module $M$, by Corollary \ref{cor:Gfd<+} we have
$$\mathrm{Gfd}_{RG}M \leq \mathrm{Ghd}_RG + {\rm Gfd}_RM.$$
Then, $\underline{{\rm hd}}_RG = {\rm sup}\{ {\rm Gfd}_{RG}M~~|~~ M ~~R\text{-flat}\} \leq {\rm Ghd}_RG$. Consequently, we obtain the desired equality  ${\rm Ghd}_RG  = \underline{{\rm hd}}_RG$.
\end{proof}

In the following, we denote by ${\rm wgldim}R$ the weak global dimension of $R$.

\begin{prop}\label{prop:-sfliRG+}
For any group $G$ and any commutative ring $R$, there are inequalities
$$\underline{{\rm hd}}_RG \leq {\rm sfli}RG \leq \underline{{\rm hd}}_RG + {\rm wgldim}R.$$
\end{prop}

\begin{proof}
If we assume ${\rm sfli}RG < \infty$, it follows from \cite[Theorem 5.3]{Emm12} that ${\rm sfli}RG = {\rm G.wgldim}RG$, and then the first inequality is clear.

In order to prove the second inequality, it suffices to assume that both $\underline{{\rm hd}}_RG = m$ and ${\rm wgldim}R = n$ are finite.
If $n=0$, then every $R$-module is flat, and moreover, $\underline{{\rm hd}}_RG = {\rm G.wgldim}RG$ holds immediately from the definitions. Note that $\underline{{\rm hd}}_RG$ is assumed to be finite, and then we have ${\rm sfli}RG = \underline{{\rm hd}}_RG$.

Now assume $n > 0$. In this case, the argument goes by induction on the flat dimension of any $RG$-module $M$. We consider the exact sequence of $RG$-modules
$$0\longrightarrow K\longrightarrow P_{n-1}\longrightarrow \cdots P_1\longrightarrow P_0\longrightarrow M\longrightarrow 0$$
for which each $P_i$ is projective for $0\leq i<n$. Since $P_i$ is restricted to be a projective $R$-module, we get that $K$ is a flat $R$-module as ${\rm wgldim}R = n$. For any injective right $RG$-module $I$ and any $k >0$, there are isomorphisms $${\rm Tor}^{RG}_{k}(I, M) \cong {\rm Tor}^{RG}_{k-n}(I, K).$$
We infer from the assumption $\underline{{\rm hd}}_RG = m$ that ${\rm Tor}^{RG}_{k-n}(I, K) = 0$ for any $k-n > m$, and then
${\rm Tor}^{RG}_{k}(I, M) = 0$ for any $k > m+n$. Hence, we have ${\rm fd}_{(RG)^{op}}I \leq m+n$, and moreover, ${\rm sfli}(RG)^{op}\leq m+n$. Since Gorenstein weak global dimension is symmetric, it follows from \cite[Theorem 5.12]{Emm12} or \cite[Corollary 2.5]{CET21} that ${\rm sfli}RG = {\rm sfli}(RG)^{op}$, and consequently, ${\rm sfli}RG \leq m+n$. This completes the proof.
\end{proof}

\begin{cor}\label{cor:equ1}
Let $R$ be a commutative ring with finite weak global dimension. If $\underline{{\rm hd}}_RG$ is finite, then ${\rm Ghd}_RG  = \underline{{\rm hd}}_RG$.
\end{cor}

\begin{proof}
Following the assumptions, we infer from Proposition \ref{prop:-sfliRG+} that ${\rm sfli}RG \leq \underline{{\rm hd}}_RG + {\rm wgldim}R < \infty$. By Theorem \ref{thm:fGhd}, this induces ${\rm Ghd}_RG < \infty$. Then, the equality ${\rm Ghd}_RG  = \underline{{\rm hd}}_RG$ holds by Proposition \ref{prop:equ1} immediately.
\end{proof}

\begin{cor}\label{cor:equ2}
${\rm Ghd}_\mathbb{Z}G < \infty$ if and only if $\underline{{\rm hd}}_\mathbb{Z}G < \infty$. In this case, ${\rm Ghd}_\mathbb{Z}G  = \underline{{\rm hd}}_\mathbb{Z}G$.
\end{cor}

\section{Some relative homological dimensions and invariants}

Let $\Lambda$ be any associative ring with identity. We consider an acyclic complex of projective left $\Lambda$-modules
\begin{center}$\mathbf{P} = \cdots\longrightarrow P_{n+1}\longrightarrow P_{n}\longrightarrow P_{n-1}\longrightarrow\cdots$\end{center}
and a $\Lambda$-module $M\cong {\rm Ker}(P_0\rightarrow P_{-1})$. Recall that the module $M$ is \emph{Gorenstein projective} \cite{EJ00} if the complex $\mathbf{P}$ remains acyclic after applying $\mathrm{Hom}_{\Lambda}(-, Q)$ for any projective module $Q$; $M$ is called a {\em projectively coresolved Gorenstein flat module}, or a {\em PGF-module} for short \cite{SS20}, provided that the complex $\mathbf{P}$ remains acyclic after applying $I\otimes_{\Lambda}-$ for any injective right $\Lambda$-module $I$. It is clear that PGF-modules are Gorenstein flat. As shown in \cite[Theorem 4.4]{SS20}, PGF-modules are also Gorenstein projective.

Every projective module is flat, however, it is not at all clear from the definitions that Gorenstein projective modules are Gorenstein flat. It was shown in \cite[Proposition 3.4]{Hol04} that every left Gorenstein projective module is Gorenstein flat if the base ring is right coherent and has finite left finitistic dimension. In general, the relation between Gorenstein projective and Gorenstein flat modules remains somehow mysterious, and it is still an open problem that if every Gorenstein projective module is Gorenstein flat.

We have the following observation, which gives a sufficient condition for affirmative answer of Gorenstein projective-flat problem.

\begin{prop}\label{prop:obs}
Let $G$ be a group and $R$ be a commutative ring. If ${\rm Ghd}_RG<\infty$, then ${\rm sfli}R < \infty$ if and only if ${\rm sfli}RG < \infty$. Moreover, if both ${\rm Ghd}_RG$ and ${\rm sfli}R$ are finite, then every Gorenstein projective $RG$-module is a PGF-module, and furthermore is a Gorenstein flat module.
\end{prop}

\begin{proof}
The equivalence of finiteness of ${\rm sfli}R$ and ${\rm sfli}RG$ is easy to see.  The ``if'' part follows from the inequality ${\rm sfli}R \leq {\rm sfli}RG$. The ``only if'' part follows from Theorem \ref{thm:fGhd}.

Note that the finiteness of both ${\rm Ghd}_RG$ and ${\rm sfli}R$ implies ${\rm sfli}(RG)^{op} = {\rm sfli}RG < \infty$. In this case, by induction on the flat dimension of any injective right $RG$-module $I$, we infer that any acyclic complex of projective $RG$-modules remains acyclic after applying $I\otimes_{RG}-$. Hence, invoking the definitions it is immediate that every Gorenstein projective $RG$-module is a PGF-module, and moreover, is a Gorenstein flat module; see also \cite[Remark 2.3 (ii)]{Emm12}.
\end{proof}

Let $\Lambda$ be a ring. It follows from \cite[1.6]{GG87} that the finiteness of both ${\rm spli}\Lambda$ and ${\rm silp}\Lambda$ implies that ${\rm spli}\Lambda = {\rm silp}\Lambda$. If $\Lambda$ is a commutative noetherian ring, by \cite[5.9]{Jen72} one has the equality ${\rm spli}\Lambda = {\rm silp}\Lambda$. Let $R$ be a {\em commutative Gorenstein ring} (noetherian ring with finite self-injective dimension), and $G$ be any group. It follows from \cite[Theorem 4.4]{Emm10} that ${\rm spli}RG = {\rm silp}RG$.

We will consider ${\rm sfli}RG$ and ${\rm silf}RG$, and show that the comparison between ${\rm sfli}RG$ and ${\rm silf}RG$ is essentially about a problem on the relation between the projective and flat dimensions of modules. Let $\overline{\mathcal{P}}(RG)$ and $\overline{\mathcal{F}}(RG)$ denote the classes of $RG$-modules with finite projective dimension and finite flat dimension, respectively. The invariant ${\rm splf}RG$ is defined as the supremum of projective length (dimension) of flat $RG$-modules, which is due to \cite[Section II.3.3]{RG71}, while the notation ``splf'' is introduced in \cite{ET11}. See also \cite[Proposition 3.2]{CET21} for a characterization for ``splf''.

\begin{prop}\label{prop:silfRG<}
Let $R$ be a commutative Gorenstein ring, $G$ be a group. The following are equivalent:
\begin{enumerate}
\item ${\rm silf}RG < \infty$.
\item ${\rm sfli}RG < \infty$ and ${\rm splf}RG < \infty$.
\item ${\rm sfli}RG < \infty$ and $\overline{\mathcal{P}}(RG) = \overline{\mathcal{F}}(RG)$.
\end{enumerate}
\end{prop}

\begin{proof}
(1)$\Longrightarrow$(2) Since the coefficient ring $R$ is commutative Gorenstein, combining with \cite[Theorem 4.4]{Emm10} and
\cite[Proposition 2.1]{ET11}, we have ${\rm spli}RG = {\rm silp}RG = {\rm silf}RG$. Hence, we infer from ${\rm silf}RG < \infty$ that ${\rm sfli}RG \leq {\rm spli}RG < \infty$. Moreover, the finiteness of both ${\rm silf}RG$ and ${\rm spli}RG$ yields that the projective dimension of any flat $RG$-module is finite, and hence ${\rm splf}RG < \infty$.

(2)$\Longrightarrow$(1) The finiteness of both ${\rm sfli}RG$ and ${\rm splf}RG$ implies that the projective dimension of any injective $RG$-module is finite, that is, ${\rm spli}RG < \infty$. Then, it follows from \cite[Theorem 2.4]{GG87} that ${\rm silp}RG$ is also finite. Moreover, it follows from \cite[Proposition 2.1]{ET11} that ${\rm silf}RG = {\rm silp}RG < \infty$.

(2)$\Longleftrightarrow$(3) Since projective modules are flat, it is clear that $\overline{\mathcal{P}}(RG) \subseteq \overline{\mathcal{F}}(RG)$. It is easy to see that the projective dimension of any flat $RG$-module is finite if and only if $\overline{\mathcal{F}}(RG) \subseteq \overline{\mathcal{P}}(RG)$. Hence, the assertion follows.
\end{proof}

For any $\Lambda$-module $M$, the \emph{Gorenstein projective dimension}, denoted by $\mathrm{Gpd}_\Lambda M$, is defined in the standard way via resolutions by Gorenstein projective modules; see \cite{EJ00, Hol04}. The {\em Gorenstein global dimension} of $\Lambda$ \cite{BM10}, denoted by ${\rm G.gldim}\Lambda$, is defined as the supremum of Gorenstein projective dimension of all $\Lambda$-modules. Analogously, the {\em PGF-dimension} ${\rm PGF}\text{-}{\rm dim}_\Lambda M$ of any $\Lambda$-module $M$, and the {\em PGF-global dimension} ${\rm PGF}\text{-}{\rm gldim}\Lambda$ are defined. The Gorenstein projective dimension is a refinement of the PGF-dimension in the sense that ${\rm Gpd}_\Lambda M\leq {\rm PGF}\text{-}{\rm dim}_\Lambda M$ for any module $M$, and ${\rm Gpd}_\Lambda M = {\rm PGF}\text{-}{\rm dim}_\Lambda M$ if ${\rm PGF}\text{-}{\rm dim}_\Lambda M < \infty$; see \cite[Corollary 3.7]{DE22}.

\begin{cor}\label{cor:GP-F(RG)}
Let $R$ be a commutative Gorenstein ring, $G$ be any group. If ${\rm silf}RG$ is finite, then the following hold:
\begin{enumerate}
\item ${\rm Ghd}_RG < \infty$, and every Gorenstein projective $RG$-module is Gorenstein flat.
\item For any Gorenstein flat $RG$-module $M$, ${\rm PGF}\text{-}{\rm dim}_{RG}M = {\rm Gpd}_{RG}M < \infty$.
\end{enumerate}
\end{cor}

\begin{proof}
By Proposition \ref{prop:silfRG<}, the finiteness of ${\rm silf}RG$ implies both ${\rm sfli}RG < \infty$ and ${\rm splf}RG < \infty$.
Then the assertion (1) follows by Theorem \ref{thm:fGhd} and Proposition \ref{prop:obs}, and the assertion (2) is immediate from \cite[Proposition 3.9]{DE22} together with \cite[Corollary 3.7 (2)]{DE22}.
\end{proof}

\begin{cor}\label{cor:silfRG<+}
Let $R$ be a commutative Gorenstein ring. For any group $G$, the following are equivalent:
\begin{enumerate}
\item ${\rm silf}RG < \infty$.
\item ${\rm Ghd}_RG < \infty$, every Gorenstein flat $RG$-module has finite Gorenstein projective dimension and every flat Gorenstein projective $RG$-module is projective.
\end{enumerate}
\end{cor}

\begin{proof}
Note that for any commutative Gorenstein ring $R$, ${\rm spli}R = {\rm silp}R < \infty$. Then, ${\rm sfli}R \leq {\rm spli}R<\infty$, and
by Theorem \ref{thm:fGhd} we have ${\rm Ghd}_RG < \infty$ if and only if ${\rm sfli}RG<\infty$. It follows from
\cite[Proposition 3.2]{CET21} that ${\rm splf}RG<\infty$ if and only if every Gorenstein flat $RG$-module has finite Gorenstein projective dimension and every flat Gorenstein projective $RG$-module is projective. Hence, the assertion follows immediately from Proposition \ref{prop:silfRG<}.
\end{proof}

It follows from \cite[Proposition 10.2.3]{EJ00} that any Gorenstein projective module of finite projective dimension is necessarily projective. For modules over group rings, the hypothesis on the finiteness of the projective dimension may be relaxed as shown below.

\begin{cor}\label{cor:GP+F=P}
Let $R$ be a commutative Gorenstein ring, and $G$ be a group. If ${\rm silf}RG < \infty$, then any Gorenstein projective $RG$-module of finite flat dimension is necessarily projective.
\end{cor}

It is trivially true that for any finite group $G$, the integral group ring $\mathbb{Z}G$ is noetherian, while it is quite subtle for infinite $G$. It follows from \cite[Theorem 1.1]{Hall} that for any virtually polycyclic group $G$, the group ring $\mathbb{Z}G$ is noetherian. Recall that a group $G$ is polycyclic if there exists a subnormal sequence
$$G = G_1 \rhd G_2 \rhd \cdots \rhd G_n = 1$$
such that $G_k/G_{k+1}$ is cyclic for all $1 \leq k < n$. The minimal length of such a subnormal sequence is the Hirsch length of $G$. A group is virtually polycyclic, also named a polycyclic-by-finite group, if it has a polycyclic subgroup of finite index.

\begin{prop}\label{prop:silf=sfli}
Let $R$ be a commutative Gorenstein ring. Let $G$ be a group such that $RG$ is a left noetherian ring. Then ${\rm sfli}RG < \infty$ if and only if ${\rm silf}RG < \infty$. In this case, one has ${\rm sfli}RG = {\rm silf}RG$.
\end{prop}

\begin{proof}
The ``if'' part follows from Proposition \ref{prop:silfRG<}, that is, ${\rm silf}RG < \infty$ implies ${\rm sfli}RG < \infty$. By \cite[Proposition 4.2]{Emm10} we have an inequality ${\rm sfli}RG \leq {\rm silp}RG$. Moreover, it follows from
\cite[Proposition 2.1]{ET11} that ${\rm silf}RG = {\rm silp}RG$; see also \cite[Theorem 3.3]{ABHS11}. Hence, we have ${\rm sfli}RG \leq {\rm silf}RG$. Here, it is not necessary to assume that the group ring $RG$ is noetherian.

For the ``only if'' part, we assume $RG$ is a left noetherian ring, and ${\rm sfli}RG = n < \infty$. Let $F$ be any flat left $RG$-module. Consider the exact sequence of left $RG$-modules
$$0\longrightarrow F\longrightarrow I_0\longrightarrow \cdots \longrightarrow I_{n-1}\longrightarrow L\longrightarrow 0,$$
for which $I_i$ ($0\leq i\leq n-1$) are injective modules. Denote by $(-)^+$ the Pontryagin dual ${\rm Hom}_\mathbb{Z}(-, \mathbb{Q}/\mathbb{Z})$. Then we get an exact sequence of right $RG$-modules
$$0\longrightarrow L^+\longrightarrow I_{n-1}^+\longrightarrow \cdots \longrightarrow I_1^+\longrightarrow I_0^+\longrightarrow F^+\longrightarrow 0.$$
It is clear that $F^+$ is an injective $RG$-module. Since $RG$ is left noetherian, it follows that $I_i^+$ are flat right $RG$-modules for $0\leq i\leq n-1$; see for example \cite[Corollary 3.2.17]{EJ00}. Then, the assumption ${\rm sfli}RG = n$ implies that ${\rm sfli}(RG)^{op} = n$, and furthermore $L^+$ is a flat right $RG$-module. Hence, $L$ is an injective left $RG$-module by \cite[Corollary 3.2.17]{EJ00}. Consequently, we get the desired inequality ${\rm silf}RG \leq n$. This completes the proof.
\end{proof}

It is well known that ${\rm G.gldim}\Lambda < \infty$ if and only if ${\rm spli}\Lambda = {\rm silp}\Lambda < \infty$; in this case, one has ${\rm G.gldim}\Lambda = {\rm spli}\Lambda = {\rm silp}\Lambda$; see for example \cite[Theorem 4.1]{Emm12}.
For PGF-global dimension, it follows from \cite[Theorem 5.1]{DE22} that ${\rm PGF}\text{-}{\rm gldim}\Lambda <\infty$ if and only if ${\rm spli}\Lambda = {\rm silp}\Lambda < \infty$ and ${\rm sfli}\Lambda = {\rm sfli}\Lambda^{op} < \infty$; in this case one has $${\rm PGF}\text{-}{\rm gldim}\Lambda = {\rm spli}\Lambda = {\rm silp}\Lambda = {\rm G.gldim}\Lambda.$$
Moreover, we have the following.

\begin{prop}\label{prop:spli<}
Let $R$ be a commutative ring, and $G$ be a group. Then ${\rm spli}RG < \infty$ if and only if ${\rm PGF}\text{-}{\rm gldim}RG = {\rm G.gldim}RG < \infty$; in this case, ${\rm PGF}\text{-}{\rm gldim}RG = {\rm spli}RG$.
\end{prop}

\begin{proof}
It suffices to prove the ``only if'' part. It follows from \cite[Corollary 5.4]{DE22} and \cite[1.6]{GG87} that ${\rm silp}RG \leq {\rm spli}RG$ with the equality if ${\rm spli}RG<\infty$. Invoking the finiteness of ${\rm spli}RG$, we infer that ${\rm sfli}(RG)^{op} = {\rm sfli}RG < \infty$. Then the assertion and the equality ${\rm PGF}\text{-}{\rm gldim}RG = {\rm spli}RG$ follow immediately from \cite[Theorem 5.1]{DE22}.
\end{proof}

It is well known that flat dimension of any module is a refinement of its projective dimension, while it is not clear if the inequality ${\rm Gfd}M\leq {\rm Gpd}M$ holds ``locally'' for any module $M$. However, for group rings the inequality follows ``globally'' as shown below.

\begin{prop}\label{prop:Gwd<Ggd}
Let $R$ be an commutative ring  and $G$ a group. Then there is an inequality
\begin{center}${\rm G.wgldim}RG \leq {\rm PGF}\text{-}{\rm gldim}RG = {\rm G.gldim}RG \leq {\rm G.wgldim}RG + {\rm splf}RG$.\end{center}
\end{prop}

\begin{proof}
The inequality ${\rm G.wgldim}RG \leq {\rm G.gldim}RG$ follows from \cite[Corollary 1.2 (1)]{BM10}. For completeness, we briefly include an argument in \cite[Remark 5.4 (ii)]{Emm12}, which is different from that of \cite{BM10}. The inequality is obvious if ${\rm G.gldim}RG=\infty$  and hence we may assume that ${\rm G.gldim}RG = n < \infty$.  Then it follows from \cite[Theorem 4.1]{Emm12} that ${\rm silp}RG = {\rm spli}RG = n$. By the definition, we have ${\rm sfli}RG \leq {\rm spli}RG$. Hence, it follows from
\cite[Theorem 5.3]{Emm12} that ${\rm G.wgldim}RG = {\rm sfli}RG \leq n$.

The inequality ${\rm G.gldim}RG \leq {\rm G.wgldim}RG + {\rm splf}RG$ is immediate from \cite[Theorem 3.3]{CET21}. It remains to prove the equality ${\rm PGF}\text{-}{\rm gldim}RG = {\rm G.gldim}RG$. It follows from \cite[Theorem 4.4]{SS20} that every PGF-module is Gorenstein projective, and then the inequality  ${\rm G.gldim}RG \leq {\rm PGF}\text{-}{\rm gldim}RG$ is clear. Conversely, in proving the inequality ${\rm G.gldim}RG \geq {\rm PGF}\text{-}{\rm gldim}RG$, it suffices to assume that ${\rm G.gldim}RG  = n$ is finite. Then ${\rm sfli}RG \leq {\rm spli}RG = {\rm silp}RG = n$. Hence, we infer from  \cite[Theorem 5.1]{DE22} that ${\rm PGF}\text{-}{\rm gldim}RG < \infty$, and moreover ${\rm G.gldim}RG = {\rm PGF}\text{-}{\rm gldim}RG$. This completes the proof.
\end{proof}

\vskip 10pt

\noindent {\bf Acknowledgements.}\quad
The first author is grateful to Rui-Peng Zhu for his help in proving Proposition \ref{prop:GhdL=GhdG}. The authors also thank Prof. Emmanouil and  Talelli for sharing their preprint \cite{ET23}. W. Ren is supported by National Natural Science Foundation of China (No. 11871125), and G. Yang is supported by National Natural Science Foundation of China (No. 12161049) and the Natural Science Foundation of Gansu Province of China (No. 21JR7RA295).

\bibliography{}

\vskip 10pt

 {\footnotesize \noindent Wei Ren,\\
 School of Mathematical Sciences, Chongqing Normal University, Chongqing 401331, PR China\\
 }

\vskip 5pt

 {\footnotesize \noindent Gang Yang,\\
 Department of Mathematics, Lanzhou Jiaotong University, Lanzhou 730070, P.R. China\\
 }

\end{document}